\documentclass[10pt,reqno]{amsart}
\usepackage{amsmath} 
\usepackage{amssymb}
\usepackage{dsfont}
\usepackage[dvips,draft,final]{graphics}
\usepackage[T1]{fontenc}
\usepackage{fancyhdr}
\usepackage{url}

\usepackage{color}
\usepackage{graphicx}
\newtheorem{thm}{Theorem}[section]
\newtheorem{lem}[thm]{Lemma}
\newtheorem{cor}[thm]{Corollary}

\newtheorem{defn}[thm]{Definition}
\newtheorem{rmk}{Remark}

\numberwithin{equation}{section}

\newcommand{\R}{\mathbb{R}}
 
\newcommand{\N}{\mathbb{N}}

\allowdisplaybreaks
\def\epsilon{\varepsilon}
\def\phi {\varphi}

\providecommand{\abs}[1]{\left\lvert#1\right\rvert}
\providecommand{\norm}[1]{\left\lVert#1\right\rVert}

\renewcommand{\leq}{\leqslant}
\renewcommand{\geq}{\geqslant}
\providecommand{\abs}[1]{\left\lvert#1\right\rvert}
\providecommand{\norm}[1]{\left\lVert#1\right\rVert}

\begin{document}

\title[existence  of solutions for semilinear fractional wave equations]{On existence and uniqueness of solutions for semilinear fractional wave equations }

\author[Yavar Kian]{Yavar Kian}
\address{Aix-Marseille Universit\'e, CNRS, CPT UMR 7332, 13288 Marseille, France \& Universit\'e de Toulon, CNRS, CPT UMR 7332, 83957 La Garde, France}
\email{yavar.kian@univ-amu.fr}

\author[Masahiro Yamamoto]{Masahiro Yamamoto}
\address{Graduate School of Mathematical Sciences,
The University of Tokyo, 3-8-9 Komaba Meguro, Tokyo 153-8914, Japan}
\email{myama@ms.u-tokyo.ac.jp}

\maketitle

\baselineskip 18pt

\begin{abstract}  Let $\Omega$ be a  $\mathcal C^2$-bounded domain of $\R^d$, $d=2,3$, and fix $Q=(0,T)\times\Omega$ with $T\in(0,+\infty]$. In the present paper we consider a Dirichlet initial-boundary value problem associated to the semilinear fractional wave equation $\partial_t^\alpha u+\mathcal A u=f_b(u)$ in $Q$ where $1<\alpha<2$, $\partial_t^\alpha$ corresponds to the Caputo fractional derivative of order $\alpha$, $\mathcal A$ is an elliptic operator and the nonlinearity $f_b\in \mathcal C^1( \R)$ satisfies $f_b(0)=0$ and $\abs{f_b'(u)}\leq C\abs{u}^{b-1}$ for some $b>1$. We first provide a definition of local  weak solutions of this problem by applying some properties of the associated linear equation $\partial_t^\alpha u+\mathcal A u=f(t,x)$ in $Q$. Then, we prove existence of  local solutions of the semilinear fractional wave equation for some suitable values of $b>1$. Moreover, we obtain an explicit dependence of the time of existence of solutions  with respect to the initial data that allows  longer time of existence for small initial data.

\end{abstract}

\section{Introduction}
\subsection{Statement of the problem}

Let $\Omega$ be a $\mathcal C^{2}$-bounded domain of $\R^d$ with $d=2,3$. In what follows, we define $\mathcal A$ by the differential operator 
\[\mathcal Au(x)=-\sum_{i,j=1}^d\partial_{x_i}\left(a_{ij}(x)\partial_{x_j}u\right)+V(x)u(x),\quad x\in\Omega,\]
where $a_{ij}=a_{ji}\in \mathcal C^1(\overline{\Omega})$ and $V\in L^\kappa(\Omega)$, for some $\kappa>d$, satisfy
\[\sum_{i,j=1}^da_{ij}(x)\xi_i\xi_j\geq c|\xi|^2,\quad x\in\overline{\Omega},\ \xi=(\xi_1,\ldots,\xi_d)\in\R^d\]
and $V\geq 0$ a.e. in $\Omega$.

We set $T\in(0,+\infty]$,  $\Sigma=(0,T)\times\partial\Omega$ and $Q=(0,T)\times\Omega$. We consider the   following initial-boundary
value problem (IBVP in short) for the fractional semilinear wave equation
\begin{equation}\label{eq2}
\left\{\begin{aligned}\partial_t^\alpha u+\mathcal A u=f_b(u),\quad &(t,x)\in 
Q,\\  
u(t,x)=0,\quad &(t,x)\in\Sigma,\\  u(0,x)=u_0(x),\ 
\partial_tu(0,x)=u_1(x),\quad &x\in\Omega,\end{aligned}\right.
\end{equation}
where  $1<\alpha<2$, $\partial_t^\alpha$ denotes the Caputo fractional derivative with respect to $t$, 
\[
\partial_t^\alpha u(t,x):=\frac{1}{\Gamma(2-\alpha)}\int_0^t(t-s)^{1-\alpha}\partial_s^{2} u(s,x) d s,\ (t,x) \in Q,
\]
$b>1$ and $f_b\in \mathcal C^1( \R)$ satisfies $f_b(0)=0$ and 
\[\abs{f_b'(u)}\leq C\abs{u}^{b-1},\quad u\in \R.\]
The main purpose of this paper is to give a suitable definition of solutions of \eqref{eq2} and to study the well-posedeness  of this problem. 

\subsection{Physical motivations and known results}

Recall that equation \eqref{eq2} is associated to anomalous diffusion 
phenomenon. More precisely, for $1 < \alpha < 2$,
the linear part of equation \eqref{eq2} is frequently used for super-diffusive model of anomalous diffusion such as diffusion in heterogeneous media.
In particular, in the linear case (i.e., $f_b \equiv 0$), some 
physical background is found in Sokolov, Klafter and Blumen \cite{SKB}.
As for analytical results in the case of $1<\alpha<2$, we refer to 
Mainardi \cite{Ma} as one early work, and
also to \S 6.1 in Kilbas, Srivastava and Trujillo \cite{KST},
\S 10.10 in Podlubny \cite{P}.
For $0 < \alpha < 1$, we define $\partial_t^{\alpha}u$ by 
$\partial_t^\alpha u(t,x):=\frac{1}{\Gamma(1-\alpha)}\int_0^t(t-s)^{-\alpha}
\partial_s u(s,x) d s$, and there are works in view of the theory of 
partial differential equations (e.g., Beckers and Yamamoto \cite{BY},
Luchko \cite{L}, Sakamoto and Yamamoto \cite{SY}).  
Such researches are rapidly  
developing and here we do not intend to give any comprehensive lists of 
references.  

In contrast to the wave equation, even  linear fractional wave equations 
are not well studied. In fact, few authors treated the well-posedness of 
the linear IBVP associated to \eqref{eq2} and to our best knowledge even 
the definition of weak solutions does not allow source term with low 
regularity. 
For a general study of the linear fractional wave equation and the regularity 
of solutions we refer to \cite{SY}. 
When we consider e.g., reaction effects in a super-diffusive model, 
we have to introduce a semilinear term.  

To the best knowledge of the authors,
there are no publications on fractional semilinear wave equations
by the Strichartz estimate which is a common technique for semilinear 
wave and Schr\"odinger equations.
In fact, for the wave equation ($\alpha=2$), the well-posedness of problem \eqref{eq2} has been studied by various authors. In the case $\Omega=\R^{k}$ with $k\geq3$ and $\mathcal A=-\Delta$, the global well-posedness has been proved both in the subcritical case   $1<b<1+\frac{4}{k-2}$  by Ginibre and Velo 
\cite{GV1}, 
and in the critical case $b=1+\frac{4}{k-2}$ by Grillarkis \cite{G} and,
Shatah and Struwe \cite{SS1,SS2}. For $\Omega=\R^2$, Nakamura and Ozawa 
\cite{NO1,NO2} proved  global well-posedness with exponentially growing nonlinearity. Without being exhaustive, for other results related to regularity of solutions or existence of solutions for more general semilinear hyperbolic equations we refer to \cite{G,GV2,Ka,Ki, LS}. In the case of  $\Omega$ a smooth bounded domain of $\R^3$, \cite{BLP} proved the global well-posedness  in the critical case $b=5$. In addition, following the strategy set by \cite{BLP}, \cite{IJ} treated the case of  $\Omega$ a smooth bounded domain of $\R^2$ with exponentially growing nonlinearity.

\subsection{Main results}

In order to give a suitable definition of solutions of \eqref{eq2} we first need to consider  the IBVP associated to the linear fractional  wave equation
\begin{equation}\label{eq1}
\left\{\begin{aligned}\partial_t^\alpha u+\mathcal A u=f(t,x),\quad &(t,x)\in 
Q,\\  u(t,x)=0,\quad &(t,x)\in\Sigma,\\  u(0,x)=u_0(x),\ \partial_tu(0,x)
=u_1(x),\quad &x\in\Omega.\end{aligned}\right.
\end{equation}
The present paper contains three main results. Our two first main results are related to properties of solutions of \eqref{eq1}, while our last result concerns the nonlinear problem \eqref{eq2}.

 Let us first remark that in contrast to usual derivatives, 
there is no exact integration by parts formula for fractional derivatives. 
Therefore, it is difficult to introduce the definition of weak solutions of \eqref{eq1} in the sense of distributions. To overcome this gap we give the following definition of weak solutions of \eqref{eq1}. Let $\mathds{1}_{(0,T)}(t)$ be the characteristic function of $(0,T)$.

\begin{defn}\label{d1} Let $u_0\in L^2(\Omega)$, $u_1\in H^{-1}(\Omega)$ and $f\in L^1(0,T;L^2(\Omega))$. We say that problem \eqref{eq1} admits a  weak solution if there exists $v\in L^\infty_{\textrm{loc}}(\R^+;L^2(\Omega))$ such that:\\
1) $v_{\vert Q}=u$ and $\inf\{\epsilon>0:\ e^{-\epsilon t}v\in L^1(\R^+; L^2(\Omega))\}=0$,\\
2) for all $p>0$ the Laplace transform $V(p)=\int_0^{+\infty}e^{-pt}v(t,.)dt$ with respect to $t$ of $v$ solves
\[\left\{\begin{aligned}(\mathcal A +p^\alpha)V(p)=F(p)+p^{\alpha-1}u_0+p^{\alpha-2}u_1,\quad &\textrm{in }\Omega,\\  V(p)=0,\quad &\textrm{on }\partial\Omega,\end{aligned}\right.\]
where $F(p)=\mathcal L[f(t,.)\mathds{1}_{(0,T)}(t)](p)=\int_0^{T}e^{-pt}f(t,.)dt$.\end{defn}
\begin{rmk} Recall (e.g. formula (2.140) page 80 of \cite{P}) that for $h\in \mathcal C^2(\R^+)$ satisfying $\inf\{\epsilon>0:\ e^{-\epsilon t}h^{(k)}\in L^1(\R^+),\ k=0,1,2\}=\epsilon_0$
we have
\[\mathcal L[\partial^\alpha h](p)=p^\alpha H(p)-p^{\alpha-1}h(0)-p^{\alpha-2}h'(0),\quad p>\epsilon_0,\]
where $H(p)=\mathcal L[ h](p)=\int_0^{+\infty}e^{-pt}h(t)dt$. Therefore,  for sufficiently smooth data $u_0,u_1, f$ (e.g. \cite{SY}) one can check that problem \eqref{eq1} admits a unique strong solution which is also a weak solution of \eqref{eq1}.\end{rmk}

Consider the operator $A$ acting on $L^2(\Omega)$ with domain 
$D(A)=\{g\in H^1_0(\Omega):\ \mathcal Ag\in L^2(\Omega)\}$ defined by 
$Au=\mathcal A u$, $u\in D(A)$. Recall that in view of the Sobolev embedding 
theorem (e.g. \cite[Theorem 1.4.4.1]{Gr}) the multiplication operator $u\mapsto Vu$ is bounded from $H^1(\Omega)$ to 
$L^2(\Omega)$. Thus,   we have $D(A)=H^2(\Omega)\cap H^1_0(\Omega)$. 
Moreover, by $V \ge 0$ in $\Omega$, 
the operator $A$ is a strictly positive selfadjoint operator with a compact resolvent. Therefore, the spectrum of $A$ consists of a non-decreasing sequence of strictly positive eigenvalues $(\lambda_n)_{n\geq1}$. Let us also introduce 
an orthonormal  basis in the Hilbert space $L^2(\Omega)$ of eigenfunctions $(\phi_n)_{n\geq1}$ of $A$ associated to the non-decreasing sequence of  eigenvalues $(\lambda_n)_{n\geq1}$. 
From now on, by $\left\langle .,.\right\rangle$, we denote 
the scalar product of $L^2(\Omega)$.  For all $s\geq 0$, we denote by $A^s$ the operator defined by 
\[A^s h=\sum_{n=1}^{+\infty}\left\langle h,\phi_n\right\rangle \lambda_n^s\phi_n,\quad h\in D(A^s)=\left\{h\in L^2(\Omega):\ \sum_{n=1}^{+\infty}\abs{\left\langle h,\phi_n\right\rangle}^2 \lambda_n^{2s}<\infty\right\}\]
and consider on $D(A^s)$ the norm
\[\|h\|_{D(A^s)}=\left(\sum_{n=1}^{+\infty}\abs{\left\langle h,\phi_n\right\rangle}^2 \lambda_n^{2s}\right)^{\frac{1}{2}},\quad h\in D(A^s).\]
By duality, we can also set $D(A^{-s})=D(A^s)'$ by identifying 
$L^2(\Omega)' = L^2(\Omega)$ which is a Hilbert space with 
the norm 
\[\norm{h}_{D(A^{-s})}=\left(\sum_{n=1}^\infty \left\langle h,\phi_n\right\rangle_{-2s}\lambda_n^{-2s}\right)^{\frac{1}{2}}.\]
Here $\left\langle .,.\right\rangle_{-2s}$ denotes the duality bracket 
between $D(A^{-s})$ and $D(A^s)$. Since $D(A^{1/2})
=H^1_0(\Omega)$, we identify $H^{-1}(\Omega)$ with $D(A^{-1/2})$. 

Using  eigenfunction expansions  we show our first main result where we state existence and uniqueness of weak solutions of \eqref{eq1}.
\begin{thm}\label{t1} Let $u_0\in L^2(\Omega)$, $u_1\in H^{-1}(\Omega)=D(A^{-\frac{1}{2}})$, $f\in L^1(0,T; L^2(\Omega))$. Then, problem \eqref{eq1} admits a unique weak solution $u\in \mathcal C([0,T]; L^2(\Omega))$ satisfying 
\begin{equation}\label{t1a}
\norm{u}_{\mathcal C([0,T]; L^2(\Omega))}\leq C(\norm{u_0}_{L^2(\Omega)}
+\norm{u_1}_{H^{-1}(\Omega)}+\norm{f}_{L^1(0,T;L^2(\Omega))}).
\end{equation}
Moreover, assuming that there exists $0<r<\frac{1}{4}$ such that $u_0\in H^{2r}(\Omega)$, we have $u\in W^{1,1}(0,T;L^2(\Omega))$ and
\begin{equation}\label{t11b}\norm{u}_{W^{1,1}(0,T; L^2(\Omega))}\leq C(\norm{u_0}_{H^{2r}(\Omega)}+\norm{u_1}_{H^{-1}(\Omega)}+\norm{f}_{L^1(0,T;L^2(\Omega))}).\end{equation}

\end{thm}
 
Recall that for  $\gamma,r,s\geq0$, $1\leq p,q,\tilde{p},\tilde{q}\leq \infty$,  Strichartz estimates for solutions $u$ of \eqref{eq1} denotes estimates of the form
\[ \norm{u}_{\mathcal C([0,T]; H^{2r}(\Omega))}+\norm{u}_{L^p(0,T;L^q(\Omega))}\leq C(\norm{u_0}_{H^{2\gamma}(\Omega)}+\norm{u_1}_{H^{2s}(\Omega)}+\norm{f}_{L^{\tilde{p}}(0,T;L^{\tilde{q}}(\Omega))}).\]
It is well known that these estimates, introduced by \cite{St} and extended to the endpoints by \cite{KT} for both wave and Schr\"odinger equations, are  important tools in the study of well-posedness of nonlinear equations 
(e.g. \cite{BLP, G,GV2,IJ,Ki}). In the present paper we prove these estimates for solutions of \eqref{eq1}.
For this purpose, we consider $1\leq p,q\leq \infty$ and $0<\gamma<1$ satisfying:
\begin{equation}\label{st1} 
\begin{array}{l}   
1)\ \ q=\infty,\quad\textrm{for } \frac{d}{4}<\gamma<1,\\
2)\ \ 2<q<\infty,\quad\textrm{for } \gamma=\frac{d}{4},\\
3)\ \ q=\frac{2d}{d-4\gamma},\quad\textrm{for } 0<\gamma<\frac{d}{4}.\end{array}\end{equation}
\begin{equation}\label{st2}
\begin{array}{l}   
1)\ \  p<\frac{1}{1-\alpha(1-\gamma)},\quad\textrm{for }\gamma>1-\frac{1}{\alpha}, \\
2)\ \ p=\infty,\quad\textrm{for } \gamma\leq1-\frac{1}{\alpha}.
\end{array}
\end{equation}
Then, our second main result can be stated as follows.
\begin{thm}\label{t2} (Strichartz estimates) 
Assume that $1\leq p,q\leq \infty$ and $0<\gamma<1$ fulfill 
\eqref{st1}, \eqref{st2} and set 
$$
s=\max\left(0,\gamma-\frac{1}{\alpha}\right), \quad
r=\min\left(1-\frac{1}{\alpha},\gamma\right).
$$
Let $u_0\in D(A^{\gamma})$, $u_1\in D(A^{s})$, $f\in L^1(0,T; L^2(\Omega))$. 
Then, the unique weak solution $u$ of problem \eqref{eq1} is lying in 
$L^p(0,T;L^q(\Omega))\cap \mathcal C([0,T]; H^{2r}(\Omega))$ and fulfills 
estimate \begin{equation}\label{st3} \norm{u}_{\mathcal C([0,T]; H^{2r}(\Omega))}+\norm{u}_{L^p(0,T;L^q(\Omega))}\leq C(\norm{u_0}_{H^{2\gamma}(\Omega)}+\norm{u_1}_{H^{2s}(\Omega)}+\norm{f}_{L^1(0,T;L^2(\Omega))}).\end{equation}
Here the constant $C$ takes the form
\begin{equation}\label{tt2a}
C=C_0(1+T)^\delta,
\end{equation}
where  
\begin{equation}\label{del}
\delta=
\left\{\begin{array}{l}
\max\left(\alpha(1-\gamma)-1, 1-\alpha(\gamma-s),
1-\alpha(r-s), \alpha(1-r)-1\right),\ \textrm{for }p=\infty,\\
\max\left( {1\over p}, 1-\alpha(\gamma-s)+\frac{1}{p},1-\alpha(r-s), 
\alpha(1-r)-1, \alpha(1-\gamma)-1+\frac{1}{p}\right),\ \textrm{for } 
p<\infty\end{array}\right.
\end{equation}
and $C_0$ depends only on $\Omega,\gamma, d, \alpha, p$.
\end{thm}

In the last section we apply estimates \eqref{st3} to prove our last result which is related to the 
existence and uniqueness of local solutions of \eqref{eq2}. For this purpose, we first need to define local solutions of \eqref{eq2}. 
In section 2 (see also \cite{SY}), using the  eigenfunction expansions we introduce the operators
\[S_1(t)h=\sum_{k=1}^\infty E_{\alpha,1}(-\lambda_k t^\alpha)
\left\langle h,\phi_k\right\rangle\phi_n,\quad h\in L^2(\Omega),\]
\[S_2(t)h=\sum_{k=1}^\infty t E_{\alpha,2}(-\lambda_k t^\alpha)\left\langle h,
\phi_k\right\rangle\phi_n,\quad h\in L^2(\Omega),\]
\[S_3(t)h=\sum_{k=1}^\infty t^{\alpha-1} E_{\alpha,\alpha}
(-\lambda_k t^\alpha)\left\langle h,\phi_k\right\rangle\phi_k,\quad h
\in L^2(\Omega),\]
where for all $\alpha>0$, $\beta\in\R$, $E_{\alpha,\beta}$ denotes the 
Mittag-Leffler function given by
\[
E_{\alpha,\beta}(z)=\sum_{k=0}^\infty \frac{z^k}{\Gamma(\alpha k+\beta)}.
\] 
It is well known (e.g. \cite{BY,P,L,SY}) that  for all $t>0$ we have $S_j(t)
\in B(L^2(\Omega))$, $j=1,2,3$. Moreover, in view of Theorem \ref{t1}, for $u_0,u_1\in L^2(\Omega)$ and $f\in L^1(0,T;L^2(\Omega))$, 
the unique weak solution of \eqref{eq1} is given by 
\begin{equation}\label{sol} 
u(t)=S_1(t)u_0+S_2(t)u_1+\int_0^tS_3(t-s)f(s)ds.
\end{equation}
For all $T>0$, we introduce the space
\[
X_T=\mathcal C([0,T];L^2(\Omega))\cap L^{b}(0,T;L^{2b}(\Omega))
\]
with the norm
\[
\norm{v}_{X_T}=\norm{v}_{\mathcal C([0,T];L^2(\Omega))}
+ \norm{v}_{L^{b}(0,T;L^{2b}(\Omega))}.
\]
Recall that, by applying the H\"older inequality, one can check that 
for all $u, v\in X_T$ we have $f_b(u), f_b(v)\in L^1(0,T; L^2(\Omega))$ with
\begin{equation}\label{loc1}
\norm{f_b(u)}_{L^1(0,T;L^2(\Omega))}\leq C_b\norm{u}_{L^{b}(0,T;
L^{2b}(\Omega))}^b\leq C_b\norm{u}_{X_T}^b
\end{equation}
and
\begin{equation}\label{loc2} 
\norm{f_b(u)-f_b(v)}_{L^1(0,T;L^2(\Omega))}\leq C_b\norm{u-v}_{X_T}
(\norm{u}_{X_T}^{b-1}+\norm{v}_{X_T}^{b-1}),
\end{equation}
where the constant $C_b>0$ depends only on $b$, $f_b$. 
Therefore, in view of Theorem \ref{t1}, the map $\mathcal H_b$ defined by
\[
\mathcal H_b u(t)=\int_0^tS_3(t-s)f_b(u(s))ds,\quad u\in X_T
\]
is locally Lipschitz from $X_T$ to  $\mathcal C([0,T];L^2(\Omega))$.
\begin{defn} \label{d2}
Let $u_0, u_1\in L^2(\Omega)$ and $T>0$. We say that \eqref{eq2} admits a 
weak solution on $(0,T)$ if the map 
$\mathcal G_b: X_T\rightarrow \mathcal C([0,T];L^2(\Omega))$ defined by 
\[
\mathcal G_b u(t)=S_1(t)u_0+S_2(t)v_2+\int_0^tS_3(t-s)f_b(u(s))ds
\]
admits a fixed point $u\in X_T$. Such a fixed point $u\in X_T$ is called 
a weak solution to \eqref{eq2} on $(0,T)$. 
We say that problem \eqref{eq2} admits a  local weak solution if there exists 
$T>0$, depending on $u_0$, $u_1$, such that problem \eqref{eq2} admits 
a weak solution on $(0,T)$.
\end{defn}

Now we can state our result of existence and uniqueness of local solutions for 
\eqref{eq2}.
We recall that $\delta>0$ is given in \eqref{del}.

\begin{thm}\label{t4} 
Let $b> 1$ satisfy
\begin{equation}\label{t4a} 
\frac{d\alpha}{d\alpha+4(1-\alpha)}<b<\frac{d\alpha+4}{d\alpha+4(1-\alpha)}
\end{equation}
and let 
\begin{equation}\label{t4b}
\gamma=\frac{d(b-1)}{4b},\   q=2b,\  
s=\max(0,\gamma-\frac{1}{\alpha}),\  r=\min(1-\frac{1}{\alpha},\gamma),\ 
1\leq \ell <\frac{1}{2-\alpha}.
\end{equation}
Then, we can choose $p\in\left(b,\frac{1}{1-\alpha(1-\gamma)}\right)$ such that  for $u_0\in D(A^\gamma)$, $u_1\in D(A^s)$, $T_0>0$,  problem \eqref{eq2} admits a  local weak solution $u$ lying in $L^p(0,T;L^q(\Omega))\cap \mathcal C([0,T]; H^{2r}(\Omega))\cap W^{1,\ell}(0,T;L^2(\Omega))$ for some $T\leq T_0$ that takes the form
\begin{equation}\label{Tt4a}
T=\min\left[\left(\tilde{C}(\norm{u_0}_{H^{2\gamma}(\Omega)}
+ \norm{u_1}_{H^{2s}(\Omega)})\right)^{-\frac{p(b-1)}{p-b}}, T_0\right],
\end{equation}
where we set  
\begin{equation}\label{Tt4b}
\tilde{C}=\tilde{C}_0(1+T_0)^{\frac{\delta}{b-1}},
\end{equation}
and $\tilde{C}_0$ depends only on  $f_b$, $\Omega$, $\alpha$, $b$, $p$ and $d$.
 Moreover, this  local weak solution $u$ is a unique local weak 
solution of \eqref{eq2} lying in $L^p(0,T;L^q(\Omega))$ and satisfies 
\begin{equation}\label{Tt4c}
\norm{u}_{\mathcal C([0,T]; H^{2r}(\Omega))}+\norm{u}_{L^p(0,T;L^q(\Omega))}
+ \norm{u}_{W^{1,\ell}(0,T;L^2(\Omega))}\leq C(\norm{u_0}_{H^{2\gamma}(\Omega)}
+ \norm{u_1}_{H^{2s}(\Omega)}).
\end{equation}
Here the constant $C>0$ depends on $d$, $\Omega$, $f_b$, $b$, $T_0$, $p$, $\alpha$. 
\end{thm}

A direct consequence of Theorem \ref{t4} is the following.

\begin{cor}\label{c1} 
Assume that  conditions \eqref{t4a} and \eqref{t4b} are fulfilled. 
Let $u_0\in D(A^\gamma)$, $u_1\in D(A^s)$ satisfy 
\[
\left[\tilde{C}_0(\norm{u_0}_{H^{2\gamma}(\Omega)}+\norm{u_1}
_{H^{2s}(\Omega)})\right]^{-\frac{p(b-1)}{p(1+\delta)-b }}>1
\]
for some $b<p<\frac{1}{1-\alpha(1-\gamma)}$, where the constant $\tilde{C}_0$ 
is introduced in \eqref{Tt4b}. 
Then, for any $T>0$ satisfying
\begin{equation}\label{c1a}
T< \left[\tilde{C}_0(\norm{u_0}_{H^{2\gamma}(\Omega)}+\norm{u_1}
_{H^{2s}(\Omega)})\right]^{-\frac{p(b-1)}{p(1+\delta)-b }},
\end{equation}
problem \eqref{eq2} admits a  unique  weak solution $u$ on $(0,T)$ lying in 
$L^p(0,T;L^q(\Omega))\cap \mathcal C([0,T]; H^{2r}(\Omega))\cap 
W^{1,\ell}(0,T;L^2(\Omega))$.
\end{cor}

This last result means that for smaller initial data we obtain longer time 
of existence of weak solutions.

Let us remark that, this paper seems to be the first where the Definition \ref{d1} of weak solutions of \eqref{eq1} is considered. The main contribution of Definition \ref{d1} comes from the fact  that it allows well-posedness of  \eqref{eq1}  with weak conditions. Indeed, in contrast to other definitions of weak solutions for \eqref{eq1} (e.g. \cite[Definition 2.1]{SY} used by  \cite{SY} to prove existence of weak solutions of \eqref{eq1} with $f\in L^2(Q)$, $u_0\in L^2(\Omega)$, $u_1=0$ in \cite[ Corollary 2.5, 2.6]{SY}), applying Definition \ref{d1} we can show well-posedness of \eqref{eq1} with $f\in L^1(0,T;L^2(\Omega))$, $u_0\in L^2(\Omega)$ and $u_1\in H^{-1}(\Omega)$.   The choice of Definition \ref{d1} is inspired both by the analysis of \cite{P} and the connection between elliptic equations and fractional diffusion equations used by \cite{LIY}. Note also that Definition \ref{d1} plays an important role in the Definition \ref{d2} of weak solutions of \eqref{eq2}.

Let us observe that in contrast to the wave equation the solution of \eqref{eq1} are not described by a semigroup. Therefore, we can not apply many arguments that allow to improve the Strichartz estimates \eqref{st3} such as the $TT^*$ method of \cite{KT}. Nevertheless, we prove  local existence of solution of \eqref{eq2} with estimates \eqref{st3}. Note also that estimates \eqref{st3} are derived from suitable estimates of Mittag-Leffler functions.

To our best knowledge this paper is the first treating well-posedness for semilinear fractional wave equations. Contrary to semilinear wave equations, it seems difficult to give a suitable definition of the energy for \eqref{eq2}.  This is mainly due to that fact that, once again, there is no exact integration by parts formula for fractional derivatives as well as properties of composition and conjugation of the fractional Caputo derivative $\partial_t^\alpha$ (e.g. \cite[Section 2]{P}).   For this reason,   it seems complicate to derive global well-posedness from local well-posedness. However, using the explicit dependence with respect to $T$ of the constant in \eqref{st3} we can establish an explicit dependence of the time of existence $T$ of \eqref{eq2} with respect to the initial conditions $u_0$, $u_1$. From this result, we prove long time of existence for small initial data (see Corollary \ref{c1}).

\subsection{Outline}
The paper is composed of four sections.  In Section 2, we treat the well-posedness of the linear problem \eqref{eq1} and we show  Theorem 1.2. Then, in Section 3 we prove the Strichartz estimates associated to these solutions and given by  Theorem 1.3. Finally, in Section 4 we prove the local existence of solutions stated in Theorem 1.5 and Corollary 1.6.

\section{The linear equation}

The goal of this section is to prove Theorem \ref{t1}. 
For this purpose, for $k\geq 1$ we introduce $u_k\in \mathcal C(\R^+)$ 
defined by 
\begin{equation}\label{t3a} 
u_k(t)=E_{\alpha,1}(-t^\alpha\lambda_k)u_{0,k}+t E_{\alpha,2}
(-t^\alpha\lambda_k)u_{1,k}
+ \int_0^t(t-s)^{\alpha-1}E_{\alpha,\alpha}(-(t-s)^\alpha\lambda_k) f_k(s)ds,
\ t>0,
\end{equation}
where $u_{0,k}=\left\langle u_0,\phi_k\right\rangle$, $u_{1,k}=\left\langle u_1,\phi_k\right\rangle_{-1}$, $f_k(s)=\left\langle f(s),\phi_k\right\rangle
\mathds{1}_{(0,T)}(s)$.
We will show that $\underset{k\geq1}{\sum}u_k(t)\phi_k(x)$ converge to a weak 
solution of \eqref{eq1} and this weak solution is unique. Let us first recall 
the following estimates of the behavior of the Mittag-Leffler function.

\begin{lem}\label{t3} \emph{(Theorem 1.6, \cite{P})} 
If $0<\alpha<2$, $\beta\in\R$, $\pi\alpha/2<\mu<\min(\pi,\pi\alpha)$, then 
$$ 
|E_{\alpha,\beta}(z)|\leq \frac{C}{1+\abs{z}},\quad z\in\mathbb C,\ \mu
\leq|arg z|\leq \pi,
$$
where the constant $C>0$ depends only on $\alpha$, $\beta$, $\mu$.\end{lem}

Applying Lemma 2.1, one can check that, for all $t>0$ and all 
$m,n\in\mathbb N^*$, we have

\[\begin{aligned}
\|\sum_{k=m}^nu_k(t)\phi_k\|_{L^2(\Omega)}\leq& C\norm{\sum_{k=m}^nu_{0,k}\phi_k}_{L^2(\Omega)}+Ct^{1-\frac{\alpha}{2}}\norm{\sum_{k=m}^n\frac{(\lambda_kt^\alpha)^{\frac{1}{2}}}{1+\lambda_kt^\alpha}\lambda_k^{-\frac{1}{2}}u_{1,k}\phi_k}_{L^2(\Omega)}\\
&+Ct^{\alpha-1}\int_0^T\norm{\sum_{k=m}^nf_k(s)\phi_k}_{L^2(\Omega)}ds.
\end{aligned}\]
Thus, for all $T_1>0$ we obtain
\[\begin{aligned}
\sup_{t\in(0,T_1)}\|\sum_{k=m}^nu_k(t)\phi_k\|_{L^2(\Omega)}\leq& C\norm{\sum_{k=m}^nu_{0,k}\phi_k}_{L^2(\Omega)}+C(T_1)^{1-\frac{\alpha}{2}}\norm{\sum_{k=m}^n\lambda_k^{-\frac{1}{2}}u_{1,k}\phi_k}_{L^2(\Omega)}\\
&+C(T_1)^{\alpha-1}\int_0^T\norm{\sum_{k=m}^nf_k(s)\phi_k}_{L^2(\Omega)}ds\end{aligned}\]
and it follows that
\[\lim_{m,n\to\infty}\sup_{t\in(0,T_1)}\|\sum_{k=m}^nu_k(t)\phi_k\|_{L^2(\Omega)}=0.\]
Therefore, for any $T_1>0$ the serie $\underset{k\geq1}{\sum}u_k(t)\phi_k$ 
converge uniformly in $t\in(0,T_1)$ to $v\in\mathcal C(\R^+,L^2(\Omega))$. In addition, for all $N\in\mathbb N^*$ and $t>0$, we have
\begin{equation}\label{t1b}\norm{\sum_{k=1}^Nu_k(t)\phi_k}_{L^2(\Omega)}\leq C\left(\|u_0\|_{L^2(\Omega)}+t^{1-\frac{\alpha}{2}}\|u_1\|_{D(A^{-\frac{1}{2}})}+Ct^{\alpha-1}\|f\|_{L^1(0,T;L^2(\Omega))}\right).\end{equation}
Here and henceforth $\N^*$ denotes the set of all the natural number $>0$. Therefore, we deduce
\[\inf\{\epsilon>0:\ e^{-\epsilon t}v\in L^1(\R^+; L^2(\Omega))\}=0 \]
and \eqref{t1b} implies that, for all $N\in\mathbb N^*$, $t>0$ and $p>0$, 
we obtain
\[\norm{\sum_{k=1}^Ne^{-pt}u_k(t)\phi_k}_{L^2(\Omega)}\leq C\left(e^{-pt}\|u_0\|_{L^2(\Omega)}+e^{-pt}t^{1-\frac{\alpha}{2}}\|u_1\|_{D(A^{-\frac{1}{2}})}+Ce^{-pt}t^{\alpha-1}\|f\|_{L^1(0,T;L^2(\Omega))}\right).\]
Then, an application of Lebesgue's dominated convergence for functions taking values in $L^2(\Omega)$ yields
\[V(p,.)=\mathcal L[v(t,.)](p)=\sum_{k=1}^\infty\mathcal L[u_k](p)\phi_k=\sum_{k=1}^\infty U_k(p,.)\]
with $U_k(p,.)=\mathcal L[u_k](p)\phi_k$.
Moreover, the properties of the Lalpace transform of the Mittag-Leffler function (e.g. formula (1.80) pp 21 of \cite{P}) imply
\[U_k(p)=\frac{p^{\alpha-1}u_{0,k}+p^{\alpha-2}u_{1,k}+F_k(p)}{p^\alpha+\lambda_k}\phi_k=(A+p^\alpha)^{-1}\left[(\left\langle p^{\alpha-1}u_0 +F(p),\phi_k\right\rangle+\left\langle p^{\alpha-2}u_1,\phi_k\right\rangle_{-1})\phi_k \right]\]
with $F_k(p)=\mathcal L[f_k](p)=\left\langle F(p),\phi_k\right\rangle$.
Thus $U_k(p,.)$ solves
\[\left\{\begin{aligned}(\mathcal A +p^\alpha)U_k(p)=\left(\left\langle p^{\alpha-1}u_0 +F(p),\phi_k\right\rangle+\left\langle p^{\alpha-2}u_1,\phi_k\right\rangle_{-1}\right)\phi_k,\quad &\textrm{in }\Omega,\\  U_k(p)=0,\quad &\textrm{on }\partial\Omega.\end{aligned}\right.
\]
Combining this with the fact that $u_0,u_1, F(p,.)\in H^{-1}(\Omega)=D(A^{-\frac{1}{2}})$, we deduce that $\underset{k\geq1}{\sum}U_k(p,.)$  converge in $H^1_0(\Omega)$ to $V(p,.)$ and $\underset{k\geq1}{\sum}(\mathcal A +p^\alpha)U_k(p)$ converge in $H^{-1}(\Omega)$ to $F(p)+p^{\alpha-1}u_0+p^{\alpha-2}u_1$. 
Therefore, $V(p,.)$ solves
\begin{equation}\label{t1c}
\left\{\begin{aligned}
(\mathcal A +p^\alpha)V(p)=F(p)+p^{\alpha-1}u_0+p^{\alpha-2}u_1,\quad &\textrm{in }\Omega,\\  V(p)=0,\quad &\textrm{on }\partial\Omega.\end{aligned}\right.
\end{equation}
Thus, $u=v_{|Q}$ is a weak solution of \eqref{eq1}. This proves the existence 
of weak solutions lying in $\mathcal C ([0,T]; L^2(\Omega))$ and by the same 
way we obtain estimate \eqref{t1a}. It remains to show  that this solution is unique and,  when $u_0\in H^{2r}(\Omega)$, that it is lying in $W^{1,1}(0,T;L^2(\Omega))$ and that it fulfills \eqref{t11b}.

We first prove the uniqueness of solutions. Let $v_1,v_2$ be two weak solutions of \eqref{eq1}. Then, for $j=1,2$,  there exist $w_j\in L^\infty_{\textrm{loc}}(\R^+;L^2(\Omega))$ such that: ${w_j}_{\vert Q}=u_j$, $\inf\{\epsilon>0:\ e^{-\epsilon t}w_j\in L^1(\R^+; L^2(\Omega))\}=0$ and, for all $p>0$, the Laplace transform $W_j(p)$ with respect to $t$ of $w_j$ solves \eqref{t1c}. Let $p>0$ and set $W(p)=W_1(p)-W_2(p)\in L^2(\Omega)$ and note that $W(p)$ solves
\[\left\{\begin{aligned}(\mathcal A +p^\alpha)W(p)=0,\quad &\textrm{in }\Omega,\\  W(p)=0,\quad &\textrm{on }\partial\Omega.\end{aligned}\right.\]
The uniqueness of the solution of this elliptic problem implies that $W(p)=0$. Therefore, for all $p>0$ we have $W_1(p)=W_2(p)$ which implies that $w_1=w_2$ and by the same way $v_1={w_1}_{\vert Q}={w_2}_{\vert Q}=v_2$. 
This proves the uniqueness.

From now on we assume that  $u_0\in H^{2r}(\Omega)$, for $r\in(0,1/4)$, and  we will show that $ u\in W^{1,1}(0,T;L^2(\Omega))$ and that it fulfills \eqref{t11b}. 
For this purpose, we establish the following lemmata.
Here we recall that $u_k, f_k, u_{0,k}, u_{1,k}$ appear in 
\eqref{t3a}.

\begin{lem}\label{l30}
For $\lambda>0$, $\alpha>0$ and positive integer $m \in \N^*$, we have
$$
\frac{d^m}{dt^m}E_{\alpha,1}(-\lambda t^{\alpha})
= -\lambda t^{\alpha-m}E_{\alpha,\alpha-m+1}(-\lambda t^{\alpha}), \quad
t > 0
$$
and
$$
\frac{d}{dt}(tE_{\alpha,2}(-\lambda t^{\alpha}))
= E_{\alpha,1}(-\lambda t^{\alpha}), \quad t > 0.
$$
\end{lem}
\begin{proof}
The power series defining $E_{\alpha,1}(-\lambda t^{\alpha})$ and
$tE_{\alpha,2}(-\lambda t^{\alpha})$ for $t>0$ admit
the termwise differentiation any times, and the termwise differentiation 
yields the conclusions.
\end{proof}

\begin{lem}\label{l3} 
For all $k\geq1$ and $1\leq \ell <\frac{1}{2-\alpha}$, we have 
$u_k\in W^{1,\ell}(0,T)$ and 
\begin{equation}\label{l3a}
\partial_tu_k(t)=-\lambda_k t^{\alpha-1}E_{\alpha,\alpha}(-t^\alpha\lambda_k)u_{0,k}+ E_{\alpha,1}(-t^\alpha\lambda_k)u_{1,k}+\int_0^t (t-s)^{\alpha-2}E_{\alpha,\alpha-1}(-(t-s)^\alpha\lambda_k) f_k(s)ds,
\end{equation}
for a.e. $t\in (0,T)$.
\end{lem}
\begin{proof} 
First we consider the case $f_k=0$. Then, we have
\[
u_k(t)=E_{\alpha,1}(-t^\alpha\lambda_k)u_{0,k}+t E_{\alpha,2}
(-t^\alpha\lambda_k)u_{1,k},\ t>0.
\]
In view of Lemma 2.2, we see that $u_k\in \mathcal C^1([0,T])$ and
\eqref{l3a} is fulfilled. 

Second we consider the the case $u_{0,k}=u_{1,k}=0$. 
Introduce, for all $\epsilon>0$ the function
\[u_k^\epsilon(t)=\int_0^{t-\epsilon}(t-s)^{\alpha-1}E_{\alpha,\alpha}(-(t-s)^\alpha\lambda_k) f_k(s)ds,\quad 0<t<T.\]
In view of Lemma 2.2, we have $u_k^\epsilon\in W^{1,\ell}(0,T)$ and
\[
\partial_tu_k^\epsilon(t)=\epsilon^{\alpha-1} E_{\alpha,\alpha}(-\lambda_k\epsilon^\alpha)f_k(t-\epsilon)+\int_0^{t-\epsilon} (t-s)^{\alpha-2}E_{\alpha,\alpha-1}(-(t-s)^\alpha\lambda_k) f_k(s)ds,\ a.a.\ t\in(0,T).
\]
On the other hand, one can easily check that $(u_k^\epsilon)_{\epsilon>0}$ 
converge to $u_k$ as $\epsilon\to 0$ in $D'(0,T)$ and $(\partial_tu_k^\epsilon)_{\epsilon>0}$ converge to 
\[t\mapsto\int_0^{t} (t-s)^{\alpha-2}E_{\alpha,\alpha-1}(-(t-s)^\alpha\lambda_k) f_k(s)ds\]
as $\epsilon\to 0$ in $D'(0,T)$, where $D'(0,T)$ is 
the space of distributions in $(0,T)$. 
Therefore, in the sense of $D'(0,T)$ we have
\[
\partial_tu_k(t)=\int_0^{t} (t-s)^{\alpha-2}E_{\alpha,\alpha-1}
(-(t-s)^\alpha\lambda_k) f_k(s)ds,\quad 0<t<T,
\]
which implies \eqref{l3a}. In addition, applying \eqref{t3a}, we obtain
 \[\abs{\partial_tu_k(t)}\leq C\int_0^{t} (t-s)^{\alpha-2}\abs{ f_k(s)}ds.\]
Then, according to the Young inequality, we deduce that 
$\partial_tu_k\in L^l(0,T)$. 
Therefore, we have $u_k,\partial_tu_k\in L^l(0,T)$, 
which means  that $u_k\in W^{1,\ell}(0,T)$. Combining these two cases, 
we complete the proof of Lemma \ref{l3}.
\end{proof}
Let us remark that, using the fact that $0<2r<\frac{1}{2}$ and 
$D(A^{\frac{1}{2}})=H^1_0(\Omega)$, one can check by interpolation 
that $u_0\in H^{2r}(\Omega)=H^{2r}_0(\Omega)=D(A^r)$ (e.g.
\cite[Chapter 1, Theorems 11.1 and 11.6]{LM1}) and
\begin{equation}\label{t1h}
\sum_{k=1}^\infty \lambda_k^{2r}\abs{u_{0,k}}^2\leq C\norm{u_0}^2
_{H^{2r}(\Omega)}.
\end{equation}
In view of \eqref{l3a}, applying our previous arguments, for all $m,n\in\mathbb N^*$, we obtain
\[\begin{aligned}
\norm{\sum_{k=m}^n\partial_t u_k\phi_k}_{L^1(0,T;L^2(\Omega))}
\leq& C\norm{\sum_{k=m}^n\frac{(\lambda_kt^\alpha)^{1-r}}
{1+(\lambda_k^{\frac{1}{\alpha}}t)^\alpha}t^{\alpha r-1}
\lambda_k^r u_{0,k}\phi_k}_{L^1(0,T;L^2(\Omega))}\\
\ &+C\norm{\sum_{k=m}^n\frac{(\lambda_kt^\alpha)^{\frac{1}{2}}}{1+(\lambda_
k^{\frac{1}{\alpha}}t)^\alpha}t^{-\frac{\alpha}{2}}\lambda_k^{-\frac{1}{2}}
u_{1,k}\phi_k}_{L^1(0,T;L^2(\Omega))}\\
&+C\int_0^T\int_0^t(t-s)^{\alpha-2}\norm{\sum_{k=m}^nf_k(s)\phi_k}_{L^2(\Omega)}dsdt.
\end{aligned}\]
The Young inequality implies
\[
\begin{aligned}\norm{\sum_{k=m}^n\partial_t u_k\phi_k}_{L^1(0,T;L^2(\Omega))}\leq& C\frac{T^{\alpha r}}{\alpha r}\norm{\sum_{k=m}^n\lambda_k^r u_{0,k}\phi_k}_{L^2(\Omega)}+C\frac{T^{1-\frac{\alpha}{2}}}{1-\frac{\alpha}{2}}\norm{\sum_{k=m}^n\lambda_k^{-\frac{1}{2}}u_{1,k}\phi_k}_{L^2(\Omega)}\\
&+C\frac{T^{\alpha-1}}{\alpha-1}\norm{\sum_{k=m}^nf_k(s)\phi_k}_{L^1(0,T;L^2(\Omega))}.
\end{aligned}
\]
Thus, we have 
\[
\lim_{m,n\to+\infty}\norm{\sum_{k=m}^n\partial_t u_k\phi_k}
_{L^1(0,T;L^2(\Omega))}=0,
\]
which means that $\sum^n_{k=1}\partial_tu_k(t)\phi_k(x)$ is a Cauchy sequence 
and a convergent sequence in $L^1(0,T;L^2(\Omega))$. 
Since $\sum_{k=1}^n u_k(t)\phi_k(x)$ converge to $u$ in 
$\mathcal C([0,T]; L^2(\Omega))$, combining this with \eqref{t1b}, 
we deduce that $\underset{k\geq1}{\sum}u_k(t)\phi_k(x)$ converge to $u$ in $W^{1,1}(0,T;L^2(\Omega))$. Finally, repeating our previous arguments and applying \eqref{t1h}, for all $N\in\mathbb N^*$, we find
\[
\begin{aligned}\norm{\sum_{k=1}^Nu_k\phi_k}_{W^{1,1}(0,T;L^2(\Omega))}
\leq& C(\norm{u_0}_{H^{2r}(\Omega)}+\norm{u_1}_{H^{-1}(\Omega)}+\norm{f}_{L^1(0,T;L^2(\Omega))}).\end{aligned}
\]
Then, combining this estimate with \eqref{t1b} and taking the limit 
$N\to \infty$, we deduce \eqref{t1a}. 
Thus, the proof of Theorem 1.2 is completed.

\section{Strichartz estimates}

The goal of this section is to show Theorem \ref{t2}. 
We divide the proof of Theorem \ref{t2} into two steps. First we prove estimates \eqref{st3} for the weak solution $u$ of 
\eqref{eq1} with $f=0$ and then for $u_0=u_1=0$.  Henceforth $C>0$ denotes 
generic constants which are
dependent only on $\Omega$, $d$, $\alpha$, $\gamma$.

\textbf{First step:} Let $f=0$ and let $1\leq p,q\leq\infty$, $0<\gamma<1$ fulfill \eqref{st1} and \eqref{st2}. Then, \eqref{sol} implies that
\[u(t)=S_1(t)u_0+S_2(t)u_1.\]
Applying estimate Lemma 2.1, 
we deduce that for $t\mapsto S_1(t)u_0\in \mathcal C([0,T]; D(A^\gamma))\subset \mathcal C([0,T]; H^{2\gamma}(\Omega))$ with
\begin{equation}\label{t2b}
\norm{ S_1(t)u_0}_{H^{2\gamma}(\Omega)}\leq C\norm{ S_1(t)u_0}_{D(A^\gamma)}
\leq C\norm{u_0}_{D(A^\gamma)}\leq C\norm{u_0}_{H^{2\gamma}(\Omega)},\quad 
0<t<T.
\end{equation}
We have $0\le \gamma - s < 1$ by the definition of $\gamma, s$.
Therefore, in the same way, Lemma 2.1 yields that, for all $0<t<T$, we have
\[\lambda_k^{2\gamma}\abs{tE_{\alpha,2}(-\lambda_k t^\alpha)\left\langle u_1,\phi_k\right\rangle}^2\leq Ct^{2(1-(\gamma-s)\alpha)}\lambda_k^{2s}\abs{\left\langle u_1,\phi_k\right\rangle}^2\left(\frac{(\lambda_k t^\alpha)^{\gamma-s}}{1+\lambda_k t^\alpha}\right)^2.\]
Thus, for all $0<t<T$, we deduce that $S_2(t)u_1\in D(A^\gamma)\subset H^{2\gamma}(\Omega)$ with
\begin{equation}\label{t2c}
\norm{ S_2(t)u_1}_{H^{2\gamma}(\Omega)}\leq Ct^{1-(\gamma-s)\alpha}\norm{u_1}_{H^{2s}(\Omega)},\quad 0<t<T.
\end{equation}
By the Sobolev embedding theorem, for all $0<t<T$, we have $u(t,.)
\in H^{2\gamma}(\Omega) \subset L^q(\Omega)$ and
\[
\norm{u(t,.)}_{L^q(\Omega)} \leq C\norm{u(t,.)}_{H^{2\gamma}(\Omega)}
\leq C\max\left(t^{1-(\gamma-s)\alpha},1\right)(\norm{u_0}_{H^{2\gamma}(\Omega)}+\norm{u_1}_{H^{2s}(\Omega)}).
\]
On the other hand, we have $1-(\gamma-s)\alpha\geq 0$ and so 
$u\in L^\infty(0,T; L^q(\Omega))$ and  
\begin{equation}\label{t2g} 
\norm{u}_{L^p(0,T;L^q(\Omega))}\leq C(1+T)^{1-(\gamma-s)\alpha
+\frac{1}{p}}(\norm{u_0}_{H^{2\gamma}(\Omega)}+\norm{u_1}_{H^{2s}(\Omega)}).
\end{equation}
In the same way, we have
\[
\norm{ S_1(t)u_0}_{H^{2r}(\Omega)}\leq C\norm{ S_1(t)u_0}
_{H^{2\gamma}(\Omega)}\leq C\norm{u_0}_{H^{2\gamma}(\Omega)},\quad 0<t<T,
\]
\begin{equation}\label{3.4}
\norm{ S_2(t)u_1}_{H^{2r}(\Omega)}\leq C(1+T)^{1-(r-s)\alpha}\norm{u_1}
_{H^{2s}(\Omega)},\quad 0<t<T.
\end{equation}
Combining these two estimates in \eqref{3.4} with \eqref{t2g}, 
we deduce \eqref{st3} for $f=0$.

\textbf{Second step:} Let $u_0=u_1=0$. In view of Lemma 2.1, 
for all $0<t<T$, we have
\[
\lambda_k^{2\gamma}\abs{t^{\alpha-1}E_{\alpha,\alpha}(-\lambda_k t^\alpha)
\left\langle f,\phi_k\right\rangle}^2
\leq t^{2(\alpha(1-\gamma)-1)}\abs{\left\langle f,\phi_k\right\rangle}^2
\left(\frac{(\lambda_k t^\alpha)^{\gamma}}{1+\lambda_k t^\alpha}\right)^2.
\]
Thus, for all $0<t<T$ and $h\in L^2(\Omega)$, we deduce that $S_3(t)h\in D(A^\gamma)\subset H^{2\gamma}(\Omega)$ with
\[\norm{ S_3(t)h}_{H^{2\gamma}(\Omega)}\leq Ct^{\alpha(1-\gamma)-1}\norm{h}_{L^2(\Omega)},\quad 0<t<T.\]
By the Sobolev embedding theorem, for all $0<t<T$, we have $ S_3(t)h\in H^{2\gamma}(\Omega) \subset L^q(\Omega)$ with
\[\norm{S_3(t)h}_{L^q(\Omega)} \leq C\norm{S_3(t)h}_{H^{2\gamma}(\Omega)}\leq Ct^{\alpha(1-\gamma)-1}\norm{h}_{L^2(\Omega)}.\]
Applying this estimate, we obtain 
%\begin{equation}\label{t2d}
$$
\norm{u(t,.)}_{L^q(\Omega)} \leq \int_0^t\norm{S_3(t-s)f(s)}_{L^q(\Omega)}ds\leq C \int_0^t(t-s)^{\alpha(1-\gamma)-1}\norm{f(s)}_{L^2(\Omega)}ds.
$$
%\end{equation}
By $t\mapsto t^{\alpha(1-\gamma)-1}\in L^p(0,T)$, the Young inequality yields
\begin{equation}\label{t2f}
\norm{u}_{L^p(0,T;L^q(\Omega))}\leq  C\frac{T^{\alpha(1-\gamma)-1+\frac{1}{p}}}
{(p(\alpha(1-\gamma)-1)+1)^{1/p}}\norm{f}_{L^1(0,T;L^2(\Omega))}.
\end{equation}
Repeating the above arguments, we deduce that
\[\norm{u(t,.)}_{H^{2r}(\Omega)}\leq C \int_0^t(t-s)^{\alpha(1-r)-1}\norm{f(s)}_{L^2(\Omega)}ds.\]
Then, since $\alpha(1-r)-1\geq \alpha(1-(1-\alpha^{-1}))-1=0$, we deduce from the Young inequality that
\[\norm{u(t,.)}_{H^{2r}(\Omega)}\leq CT^{\alpha(1-r)-1}\norm{f}_{L^1(0,T;L^2(\Omega))}.\]
Combining this estimate with \eqref{t2g} - \eqref{t2f}, we deduce \eqref{st3} for $u_0=u_1=0$. This completes the proof of Theorem \ref{t2}.

\section{Local solutions of \eqref{eq2}}
In this section we will apply the results of the previous section to prove Theorem \ref{t4} and Corollary \ref{c1}. 
\ \\
\textbf{Proof of Theorem \ref{t4}} 
Note first that for $\gamma$ and $b$ given by \eqref{t4a} and \eqref{t4b},
we have $\gamma< \frac{d}{4}$ and
\[
\frac{d}{d-4\gamma}=b>\frac{d\alpha}{d\alpha+4(1-\alpha)},
\]
which implies by $1<\alpha<2$ and $d=2,3$ that 
\begin{equation}\label{pt4a}
\gamma>1-\frac{1}{\alpha}.
\end{equation}
On the other hand, for $1-\frac{1}{\alpha}<\gamma<\frac{d}{4}$, 
one can check that
\begin{equation}\label{tt4a} 
\gamma <\frac{d\alpha}{4+d\alpha}\Longleftrightarrow {d(b-1)\over 4b} <\frac{d\alpha}{4+d\alpha}\Longleftrightarrow b<{d\alpha+4\over d\alpha+4(1-\alpha)}.
\end{equation}
Therefore, $\gamma$ given by \eqref{t4b}  fulfills
\[
1-\frac{1}{\alpha}<\gamma<\frac{d\alpha}{4+d\alpha},
\]
which yields
$$
\frac{1}{1-\alpha\left(1-\frac{d\alpha}{4+d\alpha}\right)}
< \frac{1}{1-\alpha(1-\gamma)}.
$$
Therefore, we can choose $p$ satisfying
\[b<{d\alpha+4\over d\alpha+4(1-\alpha)}=\frac{1}{1-\alpha(1-\frac{d\alpha}{4+d\alpha})}< p<\frac{1}{1-\alpha(1-\gamma)}.\]
Moreover, for $q$ given by \eqref{t4b} we have $q=\frac{2d}{d-4\gamma}$. Thus, for $q,\gamma$ given by \eqref{t4b} and $b<p<\frac{1}{1-\alpha(1-\gamma)}$, $p,q,\gamma$  fulfill conditions \eqref{st1} and \eqref{st2} with $p>b$. 
Provided that $0 < T \le T_0$ and $M>0$ will be chosen suitably later,
we set $Y_T=L^p(0,T;L^q(\Omega))\cap \mathcal C([0,T]; H^{2r}(\Omega))$ 
and $B_M=\{u\in Y_T:\ \norm{u}_{Y_T}\leq M\}$.
Moreover, we set 
\[
\norm{u}_{Y_T}=\norm{u}_{L^p(0,T;L^q(\Omega))}+\norm{u}_{\mathcal C([0,T]; 
H^{2r}(\Omega))}.
\]
We fix the constant $C_b'>0$ which appears in estimates \eqref{loc1} and 
\eqref{loc2}.  We note that $C_b'$ is independent of $T$.
We put $C'=C_0(1+T_0)^\delta$, where the constants $C_0,\delta$ are
introduced in \eqref{tt2a}, \eqref{del} and are 
independent of $T$.
Finally we fix $C=C'(1+C_b')+1$.
Since $p>b$,  for all $u\in Y_T$ we have $u\in L^b(0,T; L^{2b}(\Omega))$. Therefore, in view of Theorem \ref{t2} and estimates \eqref{st3}, \eqref{tt2a} and \eqref{loc1}, we have $\mathcal G_b(u)\in Y_T$ and 
\begin{equation}\label{t4c}
\begin{aligned}\norm{\mathcal G_b(u)}_{Y_T}&\leq  C'(\norm{u_0}_{H^{2\gamma}(\Omega)}+\norm{u_1}_{H^{2s}(\Omega)}+\norm{f_b(u)}_{L^1(0,T;L^2(\Omega))})\\
\ &\leq C'(\norm{u_0}_{H^{2\gamma}(\Omega)}+\norm{u_1}_{H^{2s}(\Omega)}+C_b\norm{u}_{L^b(0,T;L^{2b}(\Omega))}^b)\\
\ &\leq C(\norm{u_0}_{H^{2\gamma}(\Omega)}+\norm{u_1}_{H^{2s}(\Omega)}+\norm{u}_{L^b(0,T;L^{2b}(\Omega))}^b).\end{aligned}
\end{equation}
On the other hand, by the H\"older inequality one can check that
\[
\int_0^T\norm{u(t,.)}_{L^q(\Omega)}^bdt\leq \left(\int_0^T\norm{u(t,.)}_{L^q(\Omega)}^pdt\right)^{\frac{b}{p}}T^{1-\frac{b}{p}},
\]
which implies
\begin{equation}\label{t4d}
\norm{u}_{L^b(0,T;L^{2b}(\Omega))}\leq T^{\frac{p-b}{bp}}\norm{u}_{L^p(0,T; L^q(\Omega))}.
\end{equation}
Applying this estimate to \eqref{t4c}, we obtain 
\begin{equation}\label{t4e}
\norm{\mathcal G_b(u)}_{Y_T}\leq C(\norm{u_0}_{H^{2\gamma}(\Omega)}+\norm{u_1}_{H^{2s}(\Omega)}+T^{\frac{p-b}{p}}\norm{u}_{Y_T}^b).
\end{equation}
We set $M=2C(\norm{u_0}_{H^{2\gamma}(\Omega)}+\norm{u_1}_{H^{2s}(\Omega)})$ and $T=\min\left((3CM^{b-1})^{-\frac{p}{p-b}},T_0\right)$. With these values of $M$ and $T$, one can easily verify that \eqref{t4e} implies
\[\norm{\mathcal G_bu}_{Y_T}\leq M,\quad u\in B_M.\]
In the same way, applying estimates \eqref{loc2} and \eqref{st3} 
in 
\[
\mathcal G_bu-\mathcal G_bv=\int_0^tS_3(t-s)[f_b(u(s))-f_b(v(s))]ds,
\]
we obtain
\[
\norm{\mathcal G_bu-\mathcal G_bv}_{Y_T}\leq C\norm{u-v}_{L^b(0,T;L^{2b}(\Omega))}(\norm{u}_{L^b(0,T;L^{2b}(\Omega))}^{b-1}+\norm{v}_{L^b(0,T;L^{2b}(\Omega))}
^{b-1}).
\]
Then, \eqref{t4d} and the choice of $T$ imply that for every $u,v\in B_M$, 
we have
\[
\begin{aligned}
\norm{\mathcal G_bu-\mathcal G_bv}_{Y_T}&\leq CT^{\frac{p-b}{p}}\norm{u-v}_{L^p(0,T;L^{q}(\Omega))}(\norm{u}_{L^p(0,T;L^{q}(\Omega))}^{b-1}+\norm{v}_{L^p(0,T;L^{q}(\Omega))}^{b-1})\\
\ &\leq 2CM^{b-1}T^{\frac{p-b}{p}}\norm{u-v}_{Y_T}\\
\ &\leq \frac{2}{3}\norm{u-v}_{Y_T}.
\end{aligned}
\]
Therefore, $\mathcal G_b$ is a contraction from $B_M$ to $B_M$. Consequently  $\mathcal G_b$ admits a unique fixed point $u\in B_M$ which is a local weak solution of \eqref{eq2}. Moreover, from our choice of $M$ and $T$ we deduce 
\eqref{Tt4a} and \eqref{Tt4b}.

Now we show that this solution is unique in $L^p(0,T;L^q(\Omega))$. For this purpose, consider the space $Z_T=\mathcal C([0,T];L^2(\Omega))\cap L^p(0,T;L^q(\Omega))$ with the norm
\[\norm{v}_{Z_T}=\norm{v}_{\mathcal C([0,T];L^2(\Omega))}+\norm{v}_{L^p(0,T;L^q(\Omega))},\quad v\in Z_T.\]
Repeating our previous arguments we can show that  $\mathcal G_b$ is a contraction from $B_M'$ to $B_M'$ with $B_M'=\{u\in Z_T:\ \norm{u}_{Z_T}\leq M\}$. Therefore, the fixed point $u\in B_M$ of $\mathcal G_b$ is a unique local weak 
solution of \eqref{eq2} lying in $L^p(0,T;L^q(\Omega))$. Now let us show that the unique weak solution of \eqref{eq2} lying in $L^p(0,T;L^q(\Omega))$ is also lying in $W^{1,\ell}(0,T; L^2(\Omega))$ and it fulfills \eqref{Tt4c}. 
Since $\norm{u}_{Z_T}\leq M$ and 
$T=\min\left((3CM^{b-1})^{-\frac{p}{p-b}},T_0\right)$,
by \eqref{loc1}, we obtain that $f_b(u)\in L^1(0,T;L^2(\Omega))$ satisfies 
\begin{equation}\label{t4f}
\norm{f_b(u)}_{L^1(0,T;L^2(\Omega))}\leq C(\norm{u_0}_{H^{2\gamma}(\Omega)}+\norm{u_1}_{H^{2s}(\Omega)}).
\end{equation}
Now let us set 
\[
f_k(t)=\left\langle f_b(u(t)),\phi_k\right\rangle,\ u_{0,k}=\left\langle u_0,
\phi_k\right\rangle,\ u_{1,k}=\left\langle u_1,\phi_k\right\rangle.
\]
Then, in view of Lemma \ref{l3}, $u_k(t)=\left\langle u(t),\phi_k\right\rangle\in W^{1,\ell}(0,T)$ fulfills  \eqref{l3a}. Repeating the arguments used in the last part of the proof of Theorem \ref{t1}, we obtain
\[
\begin{aligned}
\norm{\sum_{k=m}^n (\partial_t u_k)\phi_k}_{L^\ell(0,T;L^2(\Omega))}
\leq& C\norm{\sum_{k=m}^n\frac{(\lambda_kt^\alpha)^{1-\gamma}}
{1+(\lambda_k^{\frac{1}{\alpha}}t)^\alpha}t^{\alpha\gamma-1}
\lambda_k^\gamma u_{0,k}\phi_k}_{L^\ell(0,T;L^2(\Omega))}
+CT^{\frac{1}{\ell}}\norm{\sum_{k=m}^nu_{1,k}\phi_k}_{L^2(\Omega)}\\
&+ C\norm{\int_0^t(t-s)^{\alpha-2}\norm{\sum_{k=m}^nf_k(s)\phi_k}
_{L^2(\Omega)}ds}_{L^\ell(0,T)}
\end{aligned}
\]
for all $m,n\in\mathbb N^*$.
In view of \eqref{pt4a}, we have
\[
\ell(\alpha\gamma-1)>\frac{\alpha-2}{2-\alpha}>-1.
\]
Therefore, the Young inequality yields
\[
\begin{aligned}
\norm{\sum_{k=m}^n(\partial_t u_k)\phi_k}_{L^\ell(0,T;L^2(\Omega))}
\leq& C\left(\frac{T^{\ell(\alpha\gamma-1)+1}}{\ell(\alpha\gamma-1)+1}\right)
^{\frac{1}{\ell}}\norm{\sum_{k=m}^nu_{0,k}\phi_k}_{L^2(\Omega)}+CT^{\frac{1}{\ell}}\norm{\sum_{k=m}^nu_{1,k}\phi_k}_{L^2(\Omega)}\\
&+ C\left(\frac{T^{\ell(\alpha-2)+1}}{\ell(\alpha-2)+1}\right)^{\frac{1}{\ell}}
\norm{\sum_{k=m}^nf_k(s)\phi_k}_{L^1(0,T;L^2(\Omega))}.
\end{aligned}
\]
Thus, we have 
\[
\lim_{m,n\to+\infty}\norm{\sum_{k=m}^n(\partial_t u_k)\phi_k}
_{L^\ell(0,T;L^2(\Omega))}=0,
\]
which means that $\underset{k\geq1}{\sum}(\partial_tu_k)(t)\phi_k(x)$ is a 
Cauchy sequence and is a convergent sequence in $L^\ell(0,T;L^2(\Omega))$. 
Combining this with the fact that $\underset{k\geq1}{\sum}u_k(t)\phi_k(x)$ 
converge to $u$ in $\mathcal C([0,T]; L^2(\Omega))$, we deduce that 
$\underset{k\geq1}{\sum}u_k(t)\phi_k(x)$ converge to $u$ in $W^{1,\ell}
(0,T;L^2(\Omega))$. Finally,  for all $N\in\mathbb N^*$, we find
\[
\begin{aligned}
\norm{\sum_{k=1}^Nu_k\phi_k}_{W^{1,\ell}(0,T;L^2(\Omega))}
\leq& C(\norm{u_0}_{H^{2\gamma}(\Omega)}+\norm{u_1}_{H^{2s}(\Omega)}+\norm{f_b(u)}_{L^1(0,T;L^2(\Omega))}).
\end{aligned}
\]
Combining this estimate with \eqref{t4f} and letting $N\to \infty$, 
we deduce \eqref{Tt4c}.  Thus, the proof of Theorem 1.5 is completed.
\qed

\textbf{Proof of Corollary \ref{c1}.} 
Let $T>0$ fulfill \eqref{c1a} and set $T_0=T$. Without lost of generality we can assume that $T\geq 1$. Then, we have
\[
\left(\tilde{C}_0T_0^{\frac{\delta}{b-1}}(\norm{u_0}_{H^{2\gamma}(\Omega)}
+\norm{u_1}_{H^{2s}(\Omega)})\right)^{-\frac{p(b-1)}{p-b}}> T_0.
\]
Since $T_0\geq 1$ we can replace $T_0$ by $T_0+1$ in condition \eqref{Tt4a}.  
Therefore, with this value of $T_0$, condition \eqref{Tt4a} holds. Thus, according to Theorem \ref{t4}, problem \eqref{eq2} admits a  unique  weak solution $u$ on $(0,T)$ lying in $L^p(0,T;L^q(\Omega))\cap \mathcal C([0,T]; H^{2r}(\Omega))\cap W^{1,\ell}(0,T;L^2(\Omega))$.
\qed

{\bf Acknolwedgements.}
The second author is partially supported by Grant-in-Aid for Scientific 
Research (S) 15H05740 of Japan Society for the Promotion of Science.

\end{document}